\documentclass[12pt]{article}
\usepackage{amsmath,amsfonts, amssymb, amsmath, graphicx}

\usepackage{latexsym}
\title{Commensurated Subgroups, Semistability  and Simple Connectivity at Infinity}
\author{G. Conner and  M. Mihalik  }
\newtheorem{theorem}{Theorem}
\newtheorem{proposition}[theorem]{Proposition}
\newtheorem{lemma}[theorem]{Lemma}
\newtheorem{corollary}[theorem]{Corollary}
\newcounter{remarknum}
\newenvironment{remark}{\addvspace{12pt}\refstepcounter{remarknum}
\noindent{\bf Remark \arabic{remarknum}.}}{\par\addvspace{12pt}}
\newenvironment{proof}{\addvspace{12pt}\noindent{\bf Proof:}}{
$\Box$\par\addvspace{12pt}}
\newcounter{examplenum}
\newenvironment{example}{\addvspace{12pt}\refstepcounter{examplenum}
\noindent{\bf Example \arabic{examplenum}.}}{\par\addvspace{12pt}}
\newcounter{definitionnum}
\newenvironment{definition}{\addvspace{12pt}\refstepcounter{definitionnum}
\noindent{\bf Definition \arabic{definitionnum}.}}{\par\addvspace{12pt}}

\date{October 6, 2011}
\begin{document}
\maketitle
\begin{abstract} 
 A subgroup $Q$ of a group $G$ is {\it commensurated} if the commensurator of $Q$ in $G$ is the entire group $G$. Our main result is that a finitely generated group $G$ containing an infinite, finitely generated, commensurated subgroup $H$, of infinite index in $G$ is 1-ended and semistable at $\infty$. If additionally, $Q$ and $G$ are finitely presented and either $Q$ is 1-ended or the pair $(G,Q)$ has 1 filtered end, then $G$ is simply connected at 
 $\infty$. A normal subgroup of a group is commensurated, so this result is a generalization of M. Mihalik's result in \cite{M1} and of B. Jackson's result in \cite{J}. As a corollary, we give an alternate proof of V. M. Lew's theorem that a finitely generated group $G$ containing an infinite, finitely generated, subnormal subgroup of infinite index is semistable at $\infty$. So, many previously known semistability and simple connectivity at $\infty$ results for group extensions follow from the results in this paper. If $\phi:H\to H$ is a monomorphism of a finitely generated group and $\phi(H)$  has finite index in $H$, then $H$ is commensurated in the corresponding ascending HNN extension, which in turn is semistable at $\infty$. 
\end{abstract}

\section{Introduction}
Given a group $G$ and a subgroup $H$ of $G$, the element $g$ of $G$ is in the {\it commensurator} of $H$ in $G$ (denoted $Comm(H,G)$) if $gHg^{-1}\cap H$ has finite index in both $H$ and $gHg^{-1}$. In the mid-1960's, A. Borel  \cite{B} proved a series of results that highlight the critical nature of commensurators in the structure of semisimple Lie groups. These results were extended by G. A. Margulis \cite{Ma} in 1975.  If $G$ is the commensurator of $Q$ in $G$, then $Q$ is {\it commensurated} in $G$. In particular, if $H$ is normal in $G$, then $H$ is commensurated in $G$. In \cite {CM} we develop the basic theory of commensurated subgroups and show this theory closely parallels the theory of normal subgroups of a group, but with subtle differences. 

A locally-finite, connected CW-complex $X$ is {\it semistable at $\infty$} if any two proper maps $r,s:[0,\infty)\to X$ which converge to the same end are properly homotopic. The early ideas of R. Lee and F. Raymond \cite {LR} and F.E.A. Johnson \cite {FEA} on the `fundamental group of an end' were instrumental in extending the idea of semistability at $\infty$ of a space to  the notion of the semistability at $\infty$ for a finitely presented group. The best reference for the fundamentals of the subject of semistability at $\infty$ is R. Geoghegan's book \cite {Ge}. Many classes of finitely generated groups are known to be semistable at $\infty$ (see  \cite{M1},  \cite{M2},  \cite{M3}, \cite{M4} and  \cite{M5} for instance). It is unknown if all finitely presented groups are semistable at $\infty$. If a finitely presented group $G$ is semistable at $\infty$, then one can define invariants for $G$, such as the fundamental group at an end of $G$ independent of choice of basepoint ray in some associated space.
The idea of semistability at $\infty$ is also of interest in the study of cohomology of groups. R. Geoghegan and M. Mihalik have shown (\cite{GM}) that if the group $G$ is finitely presented and semistable at $\infty$, then $H^2(G, \mathbb Z G)$ is free abelian. It should be noted that a basic unsolved problem in the study of group cohomology is whether or not $H^2(G,\mathbb Z G)$ is free abelian for all finitely presented groups $G$. 

The study of ends of groups was started by H. Freudenthal \cite{F} and H. Hopf \cite{H}. A finitely generated group $G$ has either $0$, $1$, $2$, or an infinite number of ends. It is elementary to see that finitely presented groups with either $0$ or $2$-ends are semistable at $\infty$. By  \cite{M3} and Dunwoody's accessibility theorem  \cite{D1} the semistability question for finitely presented groups reduces to the question of whether or not all $1$-ended finitely presented groups are semistable at $\infty$.

The strongest result to date in this subject is the following combination result.
\begin{theorem} \label{comb} {\bf (M. Mihalik, S. Tschantz \cite{comb})}
If $G=A\ast_HB$ is an amalgamated product where $A$ and $B$ are finitely presented and semistable at $\infty$, and $H$ is finitely generated, then $G$ is semistable at $\infty$. If $G=A\ast_H$ is an HNN-extension where $A$ is finitely presented and semistable at $\infty$ and $H$ is finitely generated, then $G$ is semistable at $\infty$.
\end{theorem}

This result generalizes to the obvious statement about graphs of groups and was used by Mihalik and Tschantz in  \cite{oner}, to prove that all one relator groups are semistable at $\infty$. It should be noted this result is non-trivial when $A$ and $B$ are free groups.

All word hyperbolic groups are semistable at $\infty$ (see \cite{Sw}). R. Geoghegan has shown that  a 1-ended CAT(0) group $G$ is semistable at $\infty$ if and only if  some (equivalently any) visual boundary for $G$ has the shape of a locally connected continuum  \cite{G}. It is elementary to construct a semistable at $\infty$, 1-ended CAT(0) group with non-locally connected boundary. E.g the direct product of the integers with the free group of rank 2 has visual boundary the suspension of a Cantor set - a non-locally connected space, but with the same shape as the Hawaiian earring, which is a locally connected space.  
In  \cite{M4}, a notion of semistability at $\infty$ for a finitely generated group is defined that generalizes the original definition (i.e., a finitely presented group is semistable at $\infty$ with respect to the alternative definition if and only if  it is semistable at $\infty$ with respect to the original definition). With this more general definition, the finitely generated analogs to the main results obtained in  \cite{M1} and  \cite{M2} are quite apparent. In fact, this more general definition is used to show certain finitely presented groups are semistable at $\infty$ (see \cite{M4}). In his Ph.D dissertation, Vee Ming Lew proved that 
if $G$ is a finitely generated group containing an infinite, finitely generated subnormal subgroup $H$ of infinite index in $G$, then $G$ is 1-ended and semistable at $\infty$.

Lew's proof of this theorem generalizes arguments used in the proofs in  \cite{M1} and  \cite{M2}. Our main theorem is used to produce an alternative proof of Lew's theorem. 

\begin{theorem}\label{Main}
{\bf (Main Theorem)} If a finitely generate group $G$ has an infinite finitely generated commensurated subgroup $Q$, and $Q$ has infinite index in $G$, then $G$ is one-ended and semistable at $\infty$. If additionally, $G$ and $Q$ are finitely presented and either $Q$ has one end or the pair $(G,Q)$ has one filtered end, then $G$ is simply connected at $\infty$.
\end{theorem}
As an example, the cyclic  subgroup $\langle x\rangle$ of the Baumslag-Solitar group $B(m,n)\equiv \langle x,t : t^{-1}x^mt=x^n\rangle$ (for non-zero integers $m,n$), is commensurated in $B(m,n)$. 

A connected CW-complex $X$ is {\it simply connected at $\infty$} if for each compact set $C$ in $X$ there is a compact set $D$ in $X$ such that loops in $X-D$ are homotopically trivial in $X-C$. Simple connectivity at $\infty$ implies semistability  at $\infty$. As with semistability at $\infty$, the idea of simple connectivity at $\infty$ can be extended from spaces to finitely presented groups and if $G$ is finitely presented and simply connected at $\infty$ then $H^2(G,\mathbb Z G) $ is trivial. In his thesis \cite {Si}, L. Siebenmann developed the idea of simple connectivity at $\infty$ to give an obstruction to finding a boundary for an open manifold.  In \cite{LR}, R. Lee and F. Raymond, used the idea of the simple connectivity at $\infty$ of a group in order to analyze  manifolds covered by Euclidean space.  In \cite{J}, B. Jackson proves:

\begin{theorem} \label{Ja} 
{\bf (B. Jackson)} Suppose $1\to H\to G\to K\to 1$ is a short exact sequence of infinite finitely presented groups and either $H$ or $K$ is 1-ended, then $G$ is simply connected at $\infty$. 
\end{theorem}

In \cite{Da}, M. Davis constructs examples of aspherical closed $n$-manifolds for $n\geq 4$, that are not covered by $\mathbb R^n$. In fact, Davis argues that the fundamental groups of his manifolds are semistable at $\infty$, but not simply connected at $\infty$ (and hence not covered by $\mathbb R^n$). All of Davis' group are subgroups of finite index in finitely generated Coxeter groups.  In  \cite{M5}, Mihalik shows all Artin and Coxeter groups are semistable at $\infty$.

\section{Commensurable Preliminaries}

If $S$ is a finite generating set for a group $G$, $\Gamma(G,S)$ the Cayley graph of $G$ with respect to $S$, and $H$ a subgroup of $G$, then for any $g_1, g_2 \in G$, the {\it Hausdorff} distance between $g_1H$ and $g_2H$, denoted $D_S(g_1H,g_2H)$, is the smallest integer $K$ such that for each element  $h$ of $H$ the edge path distance from $g_1h$ to $g_2H$ in $\Gamma$ is $\leq K$ and the edge path distance from $g_2h$ to $g_1H$ in $\Gamma$ is $\leq K$. If no such $K$ exists, then $D_S(g_1H,g_2H)=\infty$. 
In \cite{CM} we prove the following geometric characterization of commensurated subgroups of finitely generated groups. This characterization is the working definition of commensurated subgroup in this paper. 

\begin{proposition} \label{C8}
Suppose $S$ is a finite generating set for a group $G$ and $H$ is a subgroup of $G$, then $g\in G$ is in $Comm(H,G)$ iff the Hausdorff distance $D_S(H,gH)<\infty$ iff $D_S(H,gHg^{-1})<\infty$. 

In particular, a subgroup $Q$ of a finitely generated group $G$ is commensurated in $G$ iff the Hausdorff distance $D_S(Q,gQ)$ is finite for all $g\in G$ iff $D_S(Q,gQg^{-1})$ is finite for all $g\in G$.
\end{proposition}

Suppose $G$ is a group with finite generating set $S$ and $H$ is a subgroup of $G$. Let $\Lambda (S,H,G)$ be the graph with vertices the left cosets $gH$ of $G$ and a directed edge (labeled $s$) from $gH$ to $fH$ if for some $s\in S$ and $h_1, h_2\in H$, we have $gh_1sh_2=f$. (Equivalently, in the Cayley graph $\Gamma(S,G)$, there is an edge labeled $s$ with initial point in $gH$ and end point in $fH$.) Basically, $\Lambda$ is a (left) {\it Schreier} coset graph. Note that $\Lambda$ may have several edges labeled $s$ at a vertex. 

The following result appears in \cite{CM}.
\begin{proposition}\label{locfin2} 
Suppose $G$ is a group with finite generating set $S$ and $Q$ is commensurated in $G$. Then the graph $\Lambda(S,Q,G)$ is locally finite and $G$ acts (on the left) transitively on the vertices of $\Lambda$ and  by isometries (using the edge path metric)  on $\Lambda$. For $\Gamma(S,G)$ the Cayley graph of $G$, the projection map $p:\Gamma(S,G)\to \Lambda(S,Q,G)$ respects the action of $G$ and induces a bijection from the filtered ends of $\Gamma(S,G)$ to the ends of $\Lambda(S,Q,G)$. The graph $\Lambda(S,Q,G)$ has 0,1,2 or infinitely many ends. 
\end{proposition}

\section{Semistability Preliminaries}
Much of the groundwork for studying the notion of semistability for a finitely presented group has appeared in \cite{J}, \cite{LR}, and \cite{M1} and is well organized in \cite {G}. We will recall some of the ideas presented in these papers to set the notation for future use.

A continuous function $f:X\to Y$ is {\it proper} if for each compact subset $C$ of $Y$, $f^{-1}(C)$ is compact in $X$. A proper map $r:[0,\infty)\to X$ is called a {\it ray} in $X$.  If $K$ is a locally finite, connected CW-complex, then one can define an equivalence
relation $\sim$ on the set $A$ of all rays in $ K$ by setting $r \sim s$ if and only if
for each compact set $C \subset K$, there exists an integer $N(C)$ such that $r([N(C),\infty))$ and
$s([N(C), \infty))$ are contained in the same unbounded path component of $K -C$ (a path
component of $K-C$ is {\it unbounded} if it is not contained in any compact subset of $K$). An equivalence class of $A/\sim$ is called {\it an end of} $K$, the set of equivalence classes of $A/\sim$ is called {\it the set of ends of} $K$ and two rays in $K$, in the same equivalence class, are said to {\it converge to the same end}.
The cardinality of $A/\sim$, denoted by $e(K)$, is the {\it number of ends of} $K$.

If G is a finitely generated group with generating set $S$, then denote the {\it Cayley graph of $G$ with respect to $S$},  by $\Gamma(G,S)$. We define the {\it number of ends of $G$}, denoted by $e(G)$, to be the number of ends of the Cayley graph of $G$ with respect to a finite generating set
(i.e., $e(G) = e(\Gamma(G,S)$). This definition is independent of the choice of finite generating set for $G$.
If $G$ is finitely generated, then $e(G)$ is either 0, 1, 2, or is infinite (in which case it has the
cardinality of the real numbers). We let $\ast$ denote the basepoint of $\Gamma(G,S)$, which corresponds to the identity of $G$.

If $f, g $ are rays in $K$, then one says that $f$ and $g$ are {\it properly homotopic} if there is a proper map
$H : [0,1] \times [0,\infty) \to K$ such that
$H\vert_{\{0\}\times[0,\infty)} = f$ and $H\vert_{\{1\}\times[0,\infty)} = g$. If $f(0)=g(0)=v$, one says $f$ and $g$ are {\it properly homotopic relative to $v$} (or $rel\{v\}$) if additionally, $H\vert _{[0,1]\times \{0\}}=v$.

\begin{definition} A locally finite, connected CW-complex $K$ is {\it semistable at $\infty$} if any two rays in $K$ converging to the same end are properly homotopic. 
\end{definition}

In  \cite{M1} (Theorem 2.1) and  \cite{M2} (lemma 9) M. Mihalik proves several notions are equivalent to semistability. In \cite{M1} the space considered is simply connected, but simple connectivity is not important in that argument. Mihalik's proofs give the following result.

\begin{theorem}\label{ssequiv}
Suppose $K$ is a locally finite, connected and 1-ended CW-complex. Then the following are equivalent 
\begin{enumerate}
\item $K$ is semistable at $\infty$.
\item For any ray $r:[0,\infty )\to K$ and compact set $C$, there is a compact set $D$ such that for any third compact set $E$, and loop $\alpha$ based on $r$ and with image in $K-D$, $\alpha$ is homotopic $rel\{r\}$ to a loop in $K-E$.
\item For any compact set $C$ there is a compact set $D$ such that if $r$ and $s$ are rays based at $v$ and with image in $K-D$, then $r$ and $s$ are properly homotopic $rel\{v\}$, by a proper homotopy in $K-C$. 
\end{enumerate}
If $K$ is simply connected then a fourth equivalent condition can be added to this list:

4. If $r$ and $s$ are rays based at $v$, then $r$ and $s$ are properly homotopic 

$\ \ \ $$rel\{v\}$. 
\end{theorem}

\begin{example}
Note that the 1-ended CW-complex obtained by attaching a loop at $0$ to the interval $[0,\infty)$ is  semistable at $\infty$. Consider a ray $r$ which maps $[0,\infty)$ homeomorphically to $[0,\infty) $ and a ray $s$ which maps $[0,1]$ once around the loop and then maps  $[1,\infty)$ homeomorphically to $[0,\infty)$. Clearly $r$ and $s$ are properly homotopic, but not by a proper homotopy $rel\{0\}$. 
\end{example}
 
The following fact is proved by B. Jackson in   \cite{J}
\begin{theorem}
Suppose $X$ and $Y$ are locally finite, connected CW-complexes with $\pi_1(X)=\pi_1(Y)$. Then  the universal cover of $X$ is semistable at $\infty$ iff the universal cover of $Y$ is semistable at $\infty$.
\end{theorem}

\begin{definition} 
If $G$ is a 1-ended, finitely presented group, and $X$ is some (equivalently any) finite two dimensional CW-complex with fundamental group $G$, then we say $G$ {\it  is semistable at $\infty$ if the universal cover of $X$ is semistable at $\infty$}.
\end{definition}

We now define the notion of semistabilty for a finitely generated group as in  \cite{M4} .We give
the definition for 1-ended groups since this is what we are interested in. Suppose $G$ is a
1-ended finitely generated group with generating set $S\equiv \{g_1, g_2,\ldots , g_n\}$ and let $\Gamma
(G,S)$ be the Cayley graph of $G$ with respect to this generating set.  Suppose $\{\alpha_1, \alpha_2,\ldots , \alpha_m\}$ is a finite set of relations in $G$ written in the letters $\{g_1^\pm, g_2^\pm,\ldots , g_n^\pm\}$.  For any vertex $v\in \Gamma(G,S)$, there is an edge path cycle labeled $\alpha_i$ at $v$.   The two dimensional CW-complex $\Gamma_{(G,S)}(\alpha_1,\ldots , \alpha_m)$ is obtained by attaching to each vertex of $\Gamma(G,S)$, $2$-cells corresponding to the relations $\alpha_1,\ldots ,\alpha_n$.

In  \cite{M4} it is shown that if $S$ and $T$ are finite generating sets for the group $G$, and there are finitely many $S$-relations $P$ such that $\Gamma_{(G,S)}(P)$ is semistable at $\infty$ then there are finitely many $T$-relations $Q$ such that $\Gamma_{(G,T)}(Q)$ is semistable at $\infty$. Hence the following definition:

\begin{definition} 
We say {\it $G$ is semistable at $\infty$} if  for some finite generating set $S$ for $G$ and finite set of $S$-relations $P$, the complex $\Gamma_{(G,S)}(P)$ is semistable at $\infty$. 
\end{definition}
      
Note that if $G$ has finite presentation $\langle S:P\rangle$, then $G$ is semistable at $\infty$ with respect to definition 2 iff $G$ is semistable at $\infty$ with respect to definition 3 iff $\Gamma_{(G,S)}(P)$ is semistable at $\infty$.

Lemma 2 of  \cite{M4} is as follows:

\begin{lemma}\label{rel}
Suppose the finitely generated group $G$ is 1-ended and semistable at $\infty$. If $S$ is a finite generating set for $G$ and $P$ is a finite set of $S$-relations in $G$ such that $\Gamma _{(G,S)}(P)$ is semistable at $\infty$, then there is a finite set $Q$ of $S$ relations such that: if $r$ and $s$ are rays 
in $\Gamma_{(G,S)}(P\cup Q)$, with $r(0)=s(0)$ then $r$ is properly homotopic to $s$ $rel\{r(0)\}$. 
\end{lemma}

\begin{remark}
Using the third equivalent notion of semistability in theorem \ref{ssequiv}, it can be shown that in fact the set of relations $Q$ in the previous lemma are unnecessary in order to draw the same conclusion. I.e. If $\Gamma_{(G,S)}(P)$ is semistable at $\infty$, and $r$ and $s$ are rays 
in $\Gamma_{(G,S)}(P)$, with $r(0)=s(0)$ then $r$ is properly homotopic to $s$ $rel\{§r(0)\}$. 
\end{remark}
By an {\it edge path ray} in K, we mean a proper map $r : [0, \infty) \to K$ such that for each positive integer $n, r\vert _{[n-1,n]}$ is a homeomorphism to an edge of $K$.

If $G$ is finitely generated with finite generating set $S$, then any edge path ray, $r :  ([0,\infty),\{0\}) \to (\Gamma(G,S),*)$, can be represented as $(e_1, e_2,\ldots )$ at $*$ with $e_i \in S^\pm$, and $e_i$ the label of the $i^{th}$ edge of $r$.  Any edge path $(e_1, e_2,\ldots  ,e_k)$ of $\Gamma(G,S)$ corresponds to some group element $e_1'e_2'\ldots  e_k'$ where $e_i' \in S^\pm$.  But determining an edge path in $\Gamma(G,S)$ from some word $e_1'e_2'\ldots  e_k'$ requires a specified basepoint, since the path $(e_1', e_2',\ldots ,e_k')$ at a vertex $v$ determines a different edge path than $(e_1',e_2',\ldots  ,e_k')$ based at another vertex $w$.  Note that these edge paths differ by a covering transformation taking $v$ to $w$. By the \textsl{Star} of a subcomplex $A$ contained in a locally finite, connected CW-complex $K$, denoted \textsl{St}($A$), we mean the subcomplex of $K$ consisting of the union of all 1-cells of $K$ that intersect $A$ along with any  $n$-cell all of whose vertices lie in $St(A)$.  Note then that $A \subseteq \textsl{St}(A)$ and if $A$ is a finite subcomplex, then \textsl{St}($A$) is a finite subcomplex by the local finiteness of $K$. We recursively define the \textsl{Nth Star of $A$} for $N = 1,2,3,\ldots $ by $St^N(A) = St(St^{N-1}(A))$ where $St^0(A) = A$. When it is not clear what the over-complex might be we use the notation $St(A,K)$ to denote the {\it Star of $A$ in $K$}.

Since any ray $r : [0, \infty) \to K$ is properly homotopic to an edge path ray, we may concentrate on edge path rays when dealing with the semistability of a complex.
 
If $e$ is an edge in $K$ and $(e_1, e_2, e_3,\ldots  )$ is an edge path in $K$ based at the terminal point of $e$, then one denotes by $e * (e_1, e_2, e_3,\ldots )$ the edge path given by $e$ followed by $(e_1, e_2, e_3,\ldots )$.

\begin{definition}
For a group $G$ with finite generating set $S$ and a subset $T$ of $S$, we say an edge path in $\Gamma(G, S)$ is a {\it $T$-path} if each edge of the path is labeled by an element of  $T^{\pm}$. If the path is infinite and proper we call it a {\it  $T$-ray}.
\end{definition}

\section{Proof of Semistability in the Main Theorem}   

The 1-ended part of our main theorem is straightforward:

\begin{proposition}
Suppose $Q$ is an infinite finitely generated commensurated subgroup of infinite index in a finitely generated group $G$. Then $G$ is 1-ended.
\end{proposition}

\begin{proof}
Let $S$ be a finite generating set for $G$, containing a generating set for $Q$. Let $\Gamma\equiv \Gamma(G,S)$ and $\Gamma(Q)\equiv \Gamma(Q,S\cap Q)$. Suppose $C$ is a finite subcomplex of $\Gamma$. Only finitely many translates $g_1\Gamma(Q),\ldots , g_n\Gamma(Q)$ intersect $C$ non-trivially. Choose $D$ a finite subcomplex of $\Gamma$ such that $C\subset D$ and for each $i$, $D$ contains the bounded components of $g_i\Gamma(Q)-C$. Choose $g\in G$ such that $g\Gamma(Q)\cap C=\emptyset$. It suffices to show that for any vertex $v$ of $\Gamma-D$ there is an edge path in $\Gamma-C$ connecting $v$ to $g\Gamma(Q)$. Say $v\in h\Gamma(Q)$ and the Hausdorff distance from $hQ$ to $gQ$ in $\Gamma$ is $K$. Note that the vertices of $k\Gamma(Q)$ are $kQ$. By the choice of $D$, there is an edge path $\alpha$ in $h\Gamma (Q)-C$ from $v$ to $w\in h\Gamma(Q)-St^k(C)$. Choose a path $\beta$ of length $\leq K$ from $w$ to $g\Gamma(Q)$. Then $(\alpha, \beta)$ is a path from $v$ to $g\Gamma(Q)$ avoiding $C$. 
\end{proof}

For the remainder of the proof, $\mathcal{Q} = \{q_1, q_2,\ldots  ,q_n\}$ is a finite generating set for $Q$ and $S\equiv \{q_1, q_2,\ldots  ,q_n, k_1, k_2,\ldots ,k_t\}$ is a generating set for $G$ where $k_i\not \in Q$. Let $\mathcal K=\{k_1,\ldots ,k_t\}$. 
Our hypothesis states that for each $g\in G$, the Hausdorff distance between $Q$ and $gQ$ is finite in $\Gamma(G,S)$. 

Consider the left (Scherier) coset graph $\Lambda (S,Q,G)$ with vertex set, the set of all cosets $gQ$ in $G$. A directed edge labeled $s$ will have initial vertex $g_1Q$ and terminal vertex $g_2Q$ if there is an edge labeled $s$ in $\Gamma(G,S)$ beginning in $g_1Q$ and ending in $g_2Q$. By proposition \ref{locfin2}, $\Lambda(S,Q,G)$ is locally finite. There is a quotient map $\rho:\Gamma(G,S)\to \Lambda(S,Q,G)$ respecting the left action of $G$ on these graphs, such that each edge labeled by an element of $Q$ is mapped to a point.

\begin{lemma}\label{approx}
Suppose $S$ is a finite generating set for the group $G$ and $Q$ is a finitely generated commensurated subgroup of $G$ (with generating set a subset of $S$). There is an integer $F$ such that if $gQ$ and $hQ$ are distinct cosets (vertices) of $\Lambda(S,Q,G)$ connected by an edge labeled $s\in S^{\pm1}$, then for each $v\in gQ\subset \Gamma(S,G)$ there is a $Q$-path $\alpha$ at $v$ in $\Gamma(S,G)$ of length $< F$ such that the path $(\alpha, s)$ ends in $hQ$. 

In particular: Suppose $\alpha\equiv (e_1,e_2,\ldots)$ is an edge path (possibly infinite) at $v\in \Lambda(S,Q,G)$ (with $i^{th}$ edge labeled $e_i$) 
and $v'$ is a vertex of $\Gamma(G,S)$ such that $\rho(v')=v$ (equivalently $v'Q=v$), then there is an  edge path $\alpha'\equiv(\alpha_0',e_1,\alpha_1', e_1,\ldots )$ at $v'$ with $\alpha_i'$ a $Q$-edge path of length $< F$ such that the edge path (determined by) $\rho \alpha'$ is $\alpha$. I.e. there is $(Q,F)$- ``approximate" path lifting for $\rho$.
\end{lemma}

\begin{proof}
Suppose $v\in gQ$ and the edge labeled $s$ at $v$ ends in $hQ$. By translation, we assume $v=1\in G$ , $g=1$ and $h=s$. As $Q$ is commensurated in $G$, $sQs^{-1}\cap Q$ has finite index in $Q$. Hence there is an integer $F_s$, such that for any vertex $w\in Q$, there is a $Q$-edge path in $\Gamma(S,G)$ of length $< F_s$ from $w$ to a vertex $w'$ of $Q\cap sQs^{-1}$. As $w'\in sQs^{-1}$, $w's\in sQ$. I.e. the edge labeled $s$ at $w'$ ends in $sQ$. Let $F=max\{F_s\}_{s\in S^{\pm 1}}$.
\end{proof}



\begin{remark}
For $\alpha$ and $\alpha'$ as in Lemma \ref{approx}, we call $\alpha'$ a $(Q,F)$-{\it approximate lift of} $\alpha$. Note that lemma \ref{approx} does not imply that if $v$ and $w$ are vertices of the same coset $uQ$ then there are approximate lifts of a path $\alpha$ at $\rho (v)\in \Lambda(S,Q,G)$ to $v$ and $w$ that are $G$ translates of one another in $\Gamma(G,S)$. 
\end{remark}

The next lemma basically has the same proof as lemma 3 of  \cite{M2}.

\begin{lemma}\label{3.7} 
For each vertex $v$ of $\Lambda(S,Q,G)$, there is an edge path ray $s_v$ at $v$, such that for any finite subgraph $C$ of $\Lambda(S,Q,G)$ only finitely many $s_v$ intersect $C$. 
Furthermore, if $w\in v\equiv wQ$ let $s_w$ be a $(Q,F)$-approximate lift of $s_{\rho (w)}$ to $w\in \Gamma(G,S)$ then 

i) 
for any finite subgraph $D$ of $\Gamma(G,S)$ there are only finitely many vertices $w\in \Gamma(G,S)$ such that $s_w$ intersects $D$ non-trivially, and 

ii) for any $w\in G$, only finitely many vertices $z$ of $s_w$ are such that $zQ$ intersects $D$ non-trivially.
\end{lemma}
\begin{proof}
If $\mathcal G$ is a locally finite infinite graph then for each vertex $v$ of $\mathcal G$ there is an edge path ray $s_v$ at $v$, such that for any finite subgraph $C$ of  $\mathcal G$, only finitely many $v$ are such that $s_v$ intersects $C$. (The idea is this: Choose a base vertex $x$. For any integer $n>0$, $\mathcal G-St^n(x)$ has only finitely many components. For the finitely many vertices $v$ in $St(x)$ or a bounded component of $\mathcal G-St(x)$ choose $s_v$ to be an arbitrary edge path ray at $v$. If $v$ is a vertex of $St^2(x)$ or of a bounded component of $\mathcal G-St^2(x)$, and $s_v$ is not defined, then $v$ belongs to an unbounded component of $\mathcal G-St(x)$. Choose $s_v$ to be an edge path ray at $v$ in $\mathcal G-St(x)$. Continue.) Now pick such edge path rays for the vertices of $\Lambda(S,Q,G)$. 

As $\rho(s_w)=s_{\rho(w)}$, $s_w$ intersects $D$ iff $s_{\rho(s)}$ intersects $\rho(D)$.
Hence, we may finish the proof of  i), by showing at most finitely many vertices $v$ of a coset $gQ$ are such that $s_v$ intersects $D$. 
Otherwise, there are infinitely many distinct vertices $v_1,v_2,\ldots $ in $gQ\subset \Gamma(G,S)$ such that each edge path ray $s_{v_i}$ passes through the vertex  $d$ of $D$. 
In $\Lambda(S,Q,G)$ write the edge path ray $s_{gQ}\equiv (e_1,e_2,\ldots )$. By lemma \ref{approx}, we may write  $s_{v_i}=(\alpha_{i,1},e_1,\alpha_{i,2},e_2,\ldots )$ in $\Gamma(G,S)$, where $\alpha_{i,j}$ is a $Q$-edge path of length $< F$. Let $n(i)$ be such that  some vertex of $\alpha _{i,n(i)}$ is $d$. Since the $v_i$ are distinct and the length of each $\alpha_{i,j}$ is $< F$, the sequence of integers $\{n(1),n(2),\ldots\}$ is unbounded. But then the initial vertex of $e_{n(i)}$ (on $s_{gQ}\equiv (e_1,e_2,\ldots )$) is $\rho(d)$. This is impossible since $s_{gQ}\equiv (e_1,e_2,\ldots )$ is proper,  and $i)$ is proved.

Part ii) follows immediately from the fact that $\rho(s_w)=s_{\rho(w)}$, a proper map.
\end{proof}

By lemma \ref{approx}, if two distinct cosets $g_1Q$  and $g_2Q$ of $G$ are connected by an edge in $\Gamma (G,S)$ then they are of Hausdorff distance $\leq F$.  Choose $M$ such that  if two vertices of $Q$ in $\Gamma(G,S)$ are within $2F+1$ of one another, then their $\mathcal Q$-distance is $\leq M$. Let $P$ be the set of all $S$-relations in $G$ of length $\leq 2F+1+M$. Let $\tilde \Gamma$ be $\Gamma_{(G,S)}(P)$. 

The next result is lemma 2 of  \cite{M2}.
\begin{lemma}\label{3.1}
At each vertex $v$ of $\Gamma(G,S)$ there exists a $\mathcal{Q}$-ray $q_v$ such that for any finite subcomplex $C$ in $\Gamma(G,S)$, there are only finitely many vertices $v$ such that $q_v$ meet $C$. $\square$
\end{lemma}

For each $S$-relation $r$ of $G$, consider the  $\mathcal K$-word $r_{\mathcal K}$ obtained by eliminating from $r$, the $\mathcal Q$-letters (and their inverses). If $v$ is a vertex of $\Gamma(G,S)$ and $\alpha$ the edge path loop corresponding to $r$ at $v$, then $\rho(\alpha)$ (in $\Lambda(S,Q,G))$ has labeling $r_{\mathcal K}$. Let  $\tilde \Lambda(S,Q,G)$ be the 2-complex obtained from $\Lambda(S,Q,G)$ by attaching a 2-cell to each loop $\rho r$ (with label $r_{\mathcal K}$) where $r$ is a loop of $ \Gamma (G,S)$ of length $\leq 2F+M+1$ (only one 2-cell for a given such loop in $\Lambda(S,Q,G)$). Then $\tilde \Lambda(S,Q,G)$ is locally finite and there is a natural map $\tilde \rho:\tilde \Gamma(G,S)\to \tilde \Lambda(S,Q,G)$ extending $\rho$ and respecting the action of $G$. 

\begin{lemma}\label{3.6}
If $k \in {\mathcal K}^{\pm}$ labels an edge of $\tilde{\Gamma}$ from $v$ to $w$ and $r = (e_1, e_2, e_3,\ldots )$ is a $\mathcal{Q}$-ray at $v$, then $r$ is properly homotopic $rel\{v\}$ to $k*(f_1, f_2,\ldots )$, for $(f_1, f_2,\ldots )$ a $\mathcal{Q}$-ray at $w$, by a  homotopy $H$ with image a subset of  $St^{2F+M+1}(Im(r),\tilde \Gamma)$, and the image of $\tilde \rho \circ H$ is a subset of the finite complex $ St(\tilde\rho(k))$.
\end{lemma}
\begin{proof}
Let $v_i$ be the terminal vertex of $e_i$.  Let $v_0=v$, $w_0=w$, $\alpha _0$ be the empty path. For each $i\geq 1$, lemma \ref{approx} implies there is a $\mathcal Q$-edge path $\alpha_i $ of length $< F$ at $v_i$ so that $(\alpha_i, k)$ ends at $w_i\in kQ$. 
Note that in $\tilde \Gamma$ the distance from $w_i$ to $w_{i+1}$ is $\leq 2F+1$. For $i\geq 1$, let $f_i$ be a $\mathcal Q$-edge path in $\tilde \Gamma$ of length $\leq M$ from $w_{i-1}$ to $w_{i}$. The loop $(\alpha_i,k, f_{i+1}, k^{-1}\alpha_{i+1}^{-1}, e_{i+1}^{-1})$ has length $\leq 2F+1+M$ and so bounds a 2-cell of $\tilde \Gamma$. Hence $(e_1,e_2,\ldots ) $ is properly homotopic to $k\ast (f_1,f_2,\ldots)$ by a homotopy $H$ with image in $St^{2F+1+M}(Im(r)),\tilde \Gamma)$. As each $\alpha_i$ and each $f_i$ is a $\mathcal Q$-word, $\tilde\rho \circ H$ has image in $St(\tilde \rho (k))$. 
\end{proof}


Recall, for each vertex $v\in \tilde\Gamma$, $s_v$ is a $(Q,F)$-approximate lift of $s_{\rho(v)}$ (see lemma \ref{3.7}). 
\begin{lemma}\label{3.8} 
Suppose $D$ is a finite subcomplex of $ \tilde{\Gamma}$. Then there exists a finite complex $E_1(D) \subseteq \tilde{\Gamma}$ such that if $b = (e_1, e_2, e_3,\ldots  )$ is a $\mathcal{Q}$-ray at $v$ with image in $\tilde \Gamma-E_1(D)$, then $b$ is properly homotopic $rel\{v\}$  to $s_v$ by a homotopy in $\tilde \Gamma -D$.
\end{lemma}
\begin{proof}
Let $L=2F+M+1$ (the constant of lemma \ref{3.6}). There are only finitely many vertices $w \in \tilde\Lambda$ such that the edge path rays $s_w$ of  lemma \ref{3.7} intersect $St(\tilde\rho(D))$, non-trivially.  Call these vertices $y_1, y_2,\ldots ,y_l.$  Since each $s_{y_i}$ is proper, there are integers $J_i$ such that each edge of the ray $s_{y_i}$ following the $J_i^{th}$-edge is in $\tilde \Lambda -St(\tilde \rho(D))$. Let $J$ be the maximum $J_i$ for $i = 1,2,\ldots ,l$. By lemma \ref{3.7}, if $w$ is any vertex of $\tilde \Gamma$, and $e$ is the $j^{th}$-edge of $s_w$ for $j>FJ$, then $\tilde \rho(e)= d$ (or a vertex of $d$) for $d$ the $k^{th}$-edge of $s_{\rho(w)}$ for some $k>J$. By the definition of $J$, $d$ does not intersect $St(\tilde \rho(D))$  and so $\tilde \rho (e)$ does not intersect $St(\tilde \rho(D))$. In particular, 

$(\ast)$  If $w$ is any vertex of $\tilde \Gamma$, and $e$ is the $j^{th}$-edge of $s_w$ for $j>FJ$, then $\tilde \rho(e)\subset \tilde \Lambda-St(\tilde \rho D)$.

Let $E_1(D)$ be a compact subcomplex of  $\tilde{\Gamma}$ such that $St^{FJL}(D) \subseteq E_1(D)$ and such that $E_1(D)$ contains the finite set of vertices $v$ in $\tilde{\Gamma}$ such that $s_v$ intersects $St^{FJL}(D)$. 
Assume $b$ and $v$ satisfy the hypothesis of the lemma. The edge path ray $s_v$ (in $\tilde \Gamma-St^{FJL}(D)$) has the form $(\alpha_0,c_1, \alpha_1, c_2,\ldots )$ where $\alpha_i$ is a $\mathcal Q$-path of length $< F$ and $c_i$ is a $\mathcal K$-edge. Here $s_v$ is a $(Q,F)$-approximate lift of   $s_{\rho (v)}=(c'_1,c'_2,\ldots)$ (where $c'_i$ has the same label as $c_i$).
 
Let $v_i,w_i$ be the initial and terminal vertices of $c_i$, respectively. 
Let $b_0$ be the $\mathcal Q$-edge path ray $(\alpha_0^{-1}, b)$. By Lemma \ref{3.6}, $b_0$ is properly homotopic $rel\{v_1\}$ to $c_1 * b_1$, where $b_1$ is an $\mathcal{Q}$-ray at $w_1$, by a proper homotopy $H_1$ with image in $St^{L}(Im(b_0))$. 
In particular, $b_1$ has image in $\tilde \Gamma-St^{(FJ-1)L}(D)$. Again by Lemma \ref{3.6}, $(\alpha_1^{-1},b_1)$ is properly homotopic $rel\{v_2\}$ to $c_2*b_2$, where $b_2$ is a $\mathcal{Q}$-edge path ray, by a proper homotopy $H_2$ with image in $St^{L}(Im(b_1))\subset \tilde \Gamma-St^{(FJ-2)L}(D)$. Iterating the above process, the $\mathcal{Q}$-ray $(\alpha_j^{-1},b_j)$ is properly homotopic $rel\{v_{j+1}\}$ to $c_{j+1} * b_{j+1}$, where $b_{j+1}$ is a $\mathcal{Q}$-ray, by a proper homotopy $H_{j+1}$ with image in $St^{L}(Im(b_j))$. Let $H$ be the homotopy of $b$ to $s_v$ obtained by patching together these $H_i$. For $i\leq FJ$, $H_i$ has image in $\tilde \Gamma-D$. By lemma \ref{3.6}, $\tilde\rho \circ H_j$ has image in $St(\tilde\rho(c_{j})).$  By $(\ast)$, if $j> FJ$ then $\tilde \rho(c_j)$ misses $St(\tilde\rho(D))$. So $St(\tilde\rho(c_j))$ misses $\tilde\rho(D)$. For all positive integers $j$, $H_j$ misses $D$ and  $H$ misses $D$.

It remanins to show that $H$ is a proper. Let $C \subseteq \tilde{\Gamma}$ be a finite subcomplex. Since $\tilde\rho(s_v)$ is proper in $\tilde{\Lambda}$, there exists an integer $R$ such that if $j > R$, then $\tilde\rho (c_j)$ misses $St(\tilde\rho(C))$. As $\tilde\rho \circ H_j$ has image in $St(\tilde\rho(c_j))$,  $H_j$ misses $C$ when $j>R$. Since only finitely many of the proper  homotopies $H_j$ have image that intersect an arbitrary finite subcomplex $C$, $H$ is proper.    
\end{proof}  

\begin{lemma}\label{3.9} 
Suppose $D \subseteq \tilde{\Gamma}$ is compact. Then there exists a compact set $E_2(D) \subseteq \tilde{\Gamma}$ such that if $e$ is an edge in $\tilde{\Gamma} - E_2(D)$ from $v$ to $w$, then the $\mathcal Q$-ray $q_v$ is properly homotopic to $e*q_w$ $rel\{v\}$, by a proper homotopy in $\tilde \Gamma -D$.
\end{lemma}
\begin{proof} Again let $L=2F+M+1$ (the constant of lemma \ref{3.6}). Let $E_2(D)$ be a compact subcomplex of $ \tilde{\Gamma}$ containing $St^L(E_1(D))$ and the finite set of vertices $x$ such that $q_x$ intersects $St^L(E_1(D))$. If $e \in {\mathcal K}^{\pm1}$, then by lemma \ref{3.6}, $q_{v}$ is properly homotopic to $e*\beta$ $rel\{v\}$, where $\beta$ is a $\mathcal Q$-ray at $w$, and this homotopy has image in $St^{L}(Im(q_v))$. In particular, $\beta$ avoids $E_1(D)$.  By Lemma \ref{3.8}, $\beta$ and $q_w$ are properly homotopic $rel\{w\}$ to $s_w$ by  proper homotopies in $\tilde\Gamma-D$.  Combining these homotopies gives the result.  

If $e \in {\mathcal Q}^{\pm1}$ then lemma \ref{3.8} implies $q_v$ and $e\ast q_w$ are both properly homotopic $rel (v)$, to $s_{v}$ by a proper homotopy in $\tilde \Gamma -D$. Combining homotopies gives the desired homotopy. 
\end{proof} 

\begin{lemma}\label{3.10}
Suppose $s = (s_1, s_2, s_3,\ldots )$ is an edge path ray at a vertex $v$ in $\tilde{\Gamma}$. Then $s$ is properly homotopic to $q_v$ $rel\{v\}$. 
\end{lemma}
\begin{proof}
Choose a sequence of compact subcomplexes $\{C_i\}_{i=1}^\infty$ such that $\bigcup_{i=1}^\infty C_i = \tilde{\Gamma}$, $C_i$ is contained in the interior of $C_{i+1}$, and such that $C_{i+1}$ contains $E_2(C_i)$.  Let $v_i$ be the endpoint of $s_i$. Define $H : [0,\infty) \times [0,\infty) \to \tilde{\Gamma}$ as follows: If $R$ is the largest integer such that the edge $s_i$ misses $C_R$, then by definition of $C_R$, $q_{v_{i-1}}$ is properly homotopic $rel\{v_{i-1}\}$ to $s_i\ast q_{v_i}$ by a proper homotopy $H_i$, missing $C_{R-1}$.  Define $H$ on $[i - 1,i] \times [0, \infty)$ to be $H_i$.
 
In order to check that $H$ is proper, it suffices to show that for any compact set $C \subseteq \tilde{\Gamma}$ only finitely many $H_j$ intersect $C$. This follows from the fact that $C \subseteq C_i$ for some index $i$. Since $s$ is proper, there is an integer $W(i)$ such that for all $j \geq W_i$, $s_j$ lies in $\tilde{\Gamma} - C_{i+1}$.  So $H_j$ avoids $C$. Therefore $H$ is proper.    
\end{proof}

This completes the semistability part of our main theorem.

If $H$ is a group and $\phi:H\to H$ is a monomorphism the group with presentation $\langle t, H:t^{-1}ht \hbox{ for all } h\in H\rangle$ is called the {\it ascending} HNN extension of $H$ by $\phi$ and is denoted $H\ast_{\phi}$. 
The main theorem of \cite{HNN} states that if $H$ is a finitely presented group and $\phi:H\to H$ a monomorphism, then the ascending HNN extension $H\ast_{\phi}$ is 1-ended and semistable at $\infty$. 
Consider a general finite presentation of the form $\langle t, h_1,\ldots,h_n: r_1, \ldots ,r_n, t^{-1}h_1t=w_1,\ldots,  t^{-1}h_nt=w_n\rangle$ where $r_i$ and $w_i$ are words in $\{h_1^{\pm 1}, \ldots ,h_n^{\pm 1}\}$ for all $i$. The group $G$ of this presentation is the ascending HNN extension $H\ast_{\phi}$ where $H$ is generated by $\{h_1,\ldots, h_n\}$ and $\phi$ is the monomorphism $\phi:H\to H$, where $\phi(h_i)=w_i$ for all $i$. 
While $G$ is finitely presented it would seem rare that the finitely generated group $H$ would be finitely presented. It has long been suggested that ascending HNN extensions of this form may be a good place to search for non-semistable at $\infty$, finitely presented groups. In \cite{CM}, we show that if 
$H$ is finitely generated and the image of the monomorphism $\phi:H\to H$ has finite index in $H$, then $H$ is commensurated in $H\ast_{\phi}$. As a direct consequence of this result and our main theorem we have:

\begin{corollary} 
Suppose $H$ is a finitely generated group and $\phi:H\to H$ a monomorphism such that $\phi(H)$ has finite index in $H$ then the ascending HNN extension $H\ast_{\phi}$ is semistable at $\infty$. 
\end{corollary} 

\section{A Theorem of Lew}
Our goal in this section is to give an alternate proof of a theorem of V. M. Lew.
\begin{theorem}\label{Lew}
{\bf (V. M. Lew)} Suppose $H$ is an infinite  finitely generated subnormal subgroup of the finitely generated group $G$ and $H$ has infinite index in $G$. Then $G$ is 1-ended and semistable at $\infty$. 
\end{theorem}

\begin{proof}
Suppose $k>0$ and $H = N_0 \lhd N_1 \lhd N_2 \lhd \ldots  \lhd N_k = G$ is a subnormal series.  For $k\in\{1,2\}$ and $G$ finitely presented, theorem \ref{Lew} is proved in  \cite{M1} and \cite{M2}. Those proofs easily generalize to the finitely generated case. The result that G is 1-ended can be concluded from results in  \cite{C} or  \cite{S}. A geometric proof of this fact appears in  \cite{LR}.
We may assume that $(G:N_{k-1})=\infty$ as $G$ is semistable at $\infty$ iff any subgroup of finite index is semistable at $\infty$. 

Let $\mathcal{H} = \{h_1, h_2,\ldots  ,h_n\}$ be a finite generating set for $H$.  Now, $G$ has generating set  $S\equiv \{h_1, h_2,\ldots  ,h_n, a_1, a_2,\ldots  ,a_m, k_1, k_2,\ldots ,k_t\}$ where, under the projection map $\rho : G \to G/N_{k-1}$, $\rho(k_1),\ldots ,\rho(k_t)$ generate $G/N_{k-1}$ and the set $\{h_1,\ldots ,h_n,a_1,\ldots ,a_m\}$ is a subset of $ N_{k-1}.$  Let $\mathcal K=\{k_1,\ldots ,k_t\}$. We also assume that conjugates of the $h_i$'s by the $k_j$'s are among $a_1,\ldots ,a_m$ with the corresponding defining relations, say $k_ih_j{k_i}^{-1} \equiv a_{ij}$, and $k_i^{-1}h_jk_i \equiv b_{ij}$ for $i = 1,2,\ldots ,t$ and $j=1,2,\ldots ,n$ so that $a_{ij}, b_{ij} \in \{a_1, a_2,\ldots ,a_m\}$. Define $Q$ to be this set of conjugation relations. 
$$ Q=\{ k_ih_jk_i^{-1}a_{ij}^{-1}, k_i^{-1}h_jk_ib_{ij}^{-1}: i=1,\ldots, t \ \hbox{and} \ j=1,\ldots, n\}$$

Let $A$ be the subgroup of $N_{k-1}$ generated by $\mathcal{A} = \{h_1,\ldots ,h_n,a_1,\ldots ,a_m\}$.  
Let $A_i= N_i\cap A$ for $i\in \{1,\ldots ,n-2\}$. Then the subnormal sequence 
$$H=A_0\triangleleft A_1\triangleleft \cdots \triangleleft A_{k-2}\triangleleft A$$
has length $k-1$. The proof splits naturally into the two cases of whether or not $H$ has finite index in $A$. In the case $H$ has finite index in $A$ we give a straightforward argument showing that $H$ is commensurated in $G$ and by our main theorem $G$ is semistable at $\infty$. Note that if $k=1$ this is the only case (since $A\subset N_0=H$). So when the proof of the first case is concluded, we are in position to apply an induction argument (with base case in hand) to the remaining case. 

Suppose $H$ has finite index in $A$.  Each point of $\Gamma(A,\mathcal A)$ is within a bounded distance of $aH$ for any $a\in A$. In particular the Hausdorff distance between $H$ and $aH$ is bounded.

If $k\in \mathcal K^{\pm 1}$  and $z\in kH$ then $z=kh$ for some $h\in H$. Note that $khk^{-1}\in A$ (it is a product of the $a_{ij}^{\pm 1}$ or $b_{ij}^{\pm 1}$). Since $H$ has finite index in $A$, this point is close to $H$. As each point of $kH$ is close to $H$, left multiplying by $k^{-1}$ shows that each point of $H$ is close to $k^{-1}H$ for all $k\in \mathcal K^{\pm 1}$. We have $H$ is commensurated in $G$. The conditions of our main theorem are satisfied and so 
in the case $H$ has finite index in $A$, $G$ is semistable at $\infty$.  

Now suppose $H$ has infinite index in $A$. 
The subnormal sequence $H=N_0\triangleleft N_1\triangleleft \cdots \triangleleft N_{k-1}\triangleleft G$ has length $k$. Case 1 (or Mihalik's theorem  \cite{M1}) shows that if $k=1$ then $G$ is semistable at $\infty$. Inductively, we assume that if $G'$ is finitely generated and has a subnormal sequence of $H'=N_0'\triangleleft N_1'\triangleleft \cdots \triangleleft N_{k-2}'\triangleleft G'$ of length $k-1$ such that $H'$ is finitely generated and has infinite index in $G'$  then $G'$ is semistable at $\infty$. 

In our case, $H$ has infinite index in $A$, and the $k-1$ length subnormal series $H=A_0\triangleleft A_1\triangleleft \cdots \triangleleft A_{k-2}\triangleleft A$ implies that $A$ is semistable at $\infty$. Hence we may choose a finite set of $\mathcal A$-relations $P$ so that 
$\Gamma_{(A,\mathcal A)} (P)$ is semistable. By using lemma \ref{rel} or remark 1, we may assume that if $r$ and $s$ are $\mathcal A$-rays at $v$ in 
$\Gamma _{(A,\mathcal A)}(P)$ then $r$ and $s$ are properly homotopic $rel\{v\}$ in $\Gamma _{(A,\mathcal A)}(P)$. In this case, let $\tilde\Gamma$ be 
$\Gamma_{(G,S)}(P\cup Q)$ (where $Q$ is the set of conjugation relations defined at the beginning of this proof). 

If $v\in G$ (so $v$ is a vertex of $\tilde \Gamma$) and $C_v$ is a compact subcomplex of $v\Gamma_{(A,\mathcal A)}(P)\subset \tilde \Gamma$ there is a compact subcomplex $D_v$ of $ v\Gamma_{(A,\mathcal A)}(P)$ such that if $r$ and $s$ are edge path rays at $w\in v\Gamma_{(A,\mathcal A)}(P)-D_v$ then $r$ and $s$ are properly homotopic $rel\{v\}$ by a proper homotopy whose image does not intersect $C_v$. Hence if $C$ is a compact subcomplex of $\tilde\Gamma$ and we let $C_v=C\cap v\Gamma_{(A,\mathcal A)}(P)$ (for the finite set of vertices  $v$ such $C\cap v\Gamma_{(A,\mathcal A)}(P)\ne\emptyset$) and let $D=\cup D_v$, then any two $\mathcal A$-rays $r$ and $s$ at $w\in v\Gamma_{(A,\mathcal A)}(P)-D$ are properly homotopic $rel\{w\}$ in $\tilde\Gamma-C$.

We use $\mathcal H$-rays $r_v$, as defined in lemma \ref{3.1}. 

Choose a sequence of compact subcomplexes $\{C_i\}_{i=1}^\infty$ of $\tilde \Gamma$ satisfying the following conditions:
\begin{enumerate}
\item $\bigcup_{i=1}^\infty C_i= \tilde{\Gamma}$ 
\item $St(C_i)$ is contained in the interior of $C_{i+1}$, and the finite set of vertices $v$ such that $r_v$ intersects $C_i$, is a subset of $C_{i+1}$.  
\item If $r$ and $s$ are $\mathcal{A}$-rays both based at a vertex $v$ with images missing $C_i$, then $r$ and $s$ are properly homotopic $rel\{v\}$ by a proper homotopy missing $C_{i-1}$.
\end{enumerate}

For convenience define $C_i=\emptyset$ for $i<1$ and observe that conditions $(1)$, $(2)$, and $(3)$ remain valid for all $C_i$. 
The next lemma concludes the proof of the second case and the theorem. 

\begin{lemma}\label{4.3} If $v$ is a vertex of $\tilde{\Gamma}$, and $s = (s_1, s_2,\ldots )$ is an $S$-ray at $v$ then $s$ is properly homotopic to $r_v$, $rel\{v\}$.

\end{lemma}
\begin{proof}
Assume that $s$ has consecutive vertices $v=v_0, v_1,\ldots $.  By construction, if $v_j \in C_i - C_{i-1}$ then $r_{v_j}$ avoids $C_{i-1}$. Assume $j$ is the largest integer such that $C_j$ avoids $s_i$, we will show $r_{v_{i-1}}$ is properly homotopic to $s_i*r_{v_i}$ 
$rel\{v_{i-1}\}$ by a proper homotopy $H_i$ with image avoiding $C_{j-2}$. 

If $s_i \in \mathcal{A} ^{\pm 1}$, this is clear by condition $(3)$ with $H_i$ avoiding $C_{j-1}$. 
If $s_i \in {\mathcal K}^{\pm 1}$, then $s_i * r_{v_i}$ is properly homtopic $rel\{v_{i-1}\}$ to an 
$\mathcal{A}$-ray, $t_{v_{i-1}}$ (using only 2-cells arising from $Q$) and this homotopy has image in $St(Im(s_i * r_{v_i})) \subset \tilde{\Gamma} - C_{j-1}$. Since $t_{v_{i-1}}$ and $r_{v_{i-1}}$ are $\mathcal{A}$-rays 
with images avoiding $C_{j-1}$, condition (3) on the sets $C_i$ gives a proper homotopy between them $rel\{v_{i-1}\}$ whose image avoids $C_{j-2}$. Patch these two proper homotopies together to obtain $H_i$. 

Let $H$ be the homotopy  $rel\{v\}$ of $s$ to $r_v$, obtained by patching together the homotopies $H_i$. 
We need to check that $H$ is proper. Let $C \subset \tilde \Gamma$ be compact. Choose an index $j$ such that $C \subseteq C_j$. Since $s$ is a proper edge path to infinity, choose an index $N$ such that all edges after the $N^{th}$ edge of $s$ avoid $C_{j+2}$. Then for all $i>N$, $H_i$ avoids $C_j$, so $H$ is proper.
\end{proof}
This concludes the proof of the theorem. 
\end{proof}
\section{ Simple Connectivity at $\infty$}

Recall, a connected locally finite CW-complex $X$ is simply connected at $\infty$ if  for each compact set $C$ in $X$ there is a compact set $D$ in $X$ such that loops in $X-D$ are homotopically trivial in $X-C$. A group $G$ is simply connected at $\infty$ if given some, equivalently any (see theorem 3 of \cite {LR}), finite complex $X$ with $ \pi_1(X)=G$, then the universal cover of $X$ is simply connected at $\infty$.
 
If $G$ is a group and $H$ a subgroup of $G$ there are various notions for the number of ends of the pair $(G,H)$. Chapter 14 of Geoghegan's book \cite{G} gives a good account of these notions. In particular, the idea of the number of filtered ends of the pair $(G,H)$ is developed and compared to the standard number of ends of a pair. In any case, the number of filtered ends of the pair $(G,H)$ is greater than or equal to the number of standard ends of the pair. Proposition 14.5.9 of \cite{G} shows that if $H$ is a normal subgroup of $G$ then the number of ends of $G/H$, the standard number of ends of $(G,H)$ and the number of filtered ends of $(G,H)$ are all the same. In \cite{CM}, we show that if $G$ is  a group with finite generating set $S$ and $Q$ is a finitely generated commensurated subgroup of $G$, then the number of filtered ends of $(G,Q)$ equals the number of ends of  $\Lambda(S,Q,G)$.

\begin{theorem}
Suppose $G$ is a finitely presented group with finite generating set $S$, and $Q$ is a finitely presented, infinite commensurated subgroup of infinite index in $G$. If $Q$ or $\Lambda (S,Q,G)$ is 1-ended, then $G$ is simply connected at $\infty$. 
\end{theorem}

\begin{proof} 
Suppose $\mathcal P=\langle q_1,\ldots ,q_a, k_1,\ldots ,k_b:R\rangle$ is a finite presentation of the group $G$ such that the $q_i$ generate the infinite commensurated subgroup $Q$, no $k_i$ is an element of $ Q$, and $R$ contains relations $R'$ such that $\langle q_1,\ldots ,q_a:R'\rangle$ is a finite presentation of $Q$. Assume that $Q$ has infinite index in $G$. Let $X$ be the Cayley 2-complex of $\mathcal P$, $\tilde X$ the universal cover for $X$ and $\tilde X(Q,v)\subset \tilde X$  the copy of the universal cover of the Cayley 2-complex for $\langle q_1,\ldots ,q_a:R'\rangle$ containing $v$. Let $\mathcal K=\{k_1,\ldots ,k_b\}$ and $\mathcal Q=\{q_1,\ldots q_a\}$.

Let $N_1$ be an integer such that if cosets $gQ$ and $hQ$ of $G$ are connected by an edge in $\tilde X$, then the Hausdorff distance between $gQ$ and $hQ$ in $\tilde X$ is $\leq N_1$. 
For each relator $r\in R$, let $r'$ be the word obtained from $r$ by removing $\mathcal Q$ letters. For each such (non-trivial) $r'$ and edge loop in $\Lambda(S,Q,G)$ with edge label $r'$, attach a 2-cell and call the resulting locally finite 2-complex $\hat \Lambda(S,Q,G)$.  Note that $\Gamma(S,G)$ is the 1-skeleton of $\tilde X$. Extend the map  $\rho:\Gamma(S,G) \to \Lambda (S,Q,G)$ (see proposition \ref{locfin2}),   to  $\rho :\tilde X\to \hat \Lambda(S,Q,G)$. Let $C$ be a finite subcomplex of $\tilde X$. Let $d_1\geq 1$ be an integer such that for each vertex $v$ of $\rho(C)$, there is a $\mathcal K$-edge path in $\hat \Lambda (S,Q,G)$ of length $\leq d_1$ from $v$ to a vertex of $\hat \Lambda (S,Q,G)-\rho(C)$. In particular, for each vertex $v$ of $\tilde X$, there is an edge path at $v$ of length $\leq N_1d_1$ and with end point $w$ such that $\tilde X(Q,w)\cap C=\emptyset$.
For each $k\in \{k_1,\ldots , k_b\}$ assume that $Q$ and $kQ$ are within Hausdorff distance $N_1$. Choose $N_2$ so that if $q_1$ and $q_2$ are two $Q$-vertices of $\tilde X$ with the edge path distance in $\tilde X$ between $q_1$ and $q_2$ less than or equal to $ 2N_1+1$ then the edge path distance between $q_1$ and $q_2$ in $\tilde X(Q,q_1)$ is $\leq N_2$. In particular, there is a $\mathcal Q$-edge path between $q_1$ and $q_2$ of length $\leq N_2$. Choose $N_3$ such that if $\alpha$ is an edge path loop at $\ast\in \tilde X$ of length $\leq  
2N_1+N_2+1$ then $\alpha$ is homotopically trivial in $St^{N_3}(\ast)$.

\begin{lemma}\label{Qkill} 
Suppose $G$ is a finitely presented group, $Q$ is a finitely presented, infinite commensurated subgroup of infinite index in $G$, $\mathcal P$ is a presentation of $G$ as above, and $X$ is the Cayley 2-complex of $\mathcal P$. If $\alpha$ is a $\mathcal Q$-loop in $\tilde X$, with image in $\tilde X-St^{d_1N_1N_3}(C)$ then $\alpha$ is homotopically trivial in $\tilde X-C$. 
\end{lemma}
\begin{proof}
We may assume  $\ast$ is the initial vertex of $\alpha$. If $\tilde X(Q,v)\cap C=\emptyset$, then as $\alpha$ is homotopically trivial in $\tilde X(Q,v)$ and we are finished. If $\tilde X(Q,v)\cap C\ne\emptyset$, there is an 
edge path $\beta=(b_1,\ldots, b_k)$ at $v$ with $k\leq N_1d_1$ and with end point $w$ such that $\tilde X(Q,w)\cap C=\emptyset$. 
Let $v\equiv v_0,\ldots ,v_k\equiv w$ be the consecutive vertices of $\alpha$.
 For each vertex $x$ of $\alpha$, there is an edge path of length $\leq N_1$ from $x$ to a vertex of $\tilde X(Q,v_1)$ (if $b_1$ is a $\mathcal Q$-edge, this path is trivial) and hence $\alpha$ is homotopic $rel\{0,1\}$ to a loop $(b_1,\alpha_1,b_1^{-1})$, where $\alpha_1$ is a $\mathcal Q$-loop in $\tilde X(Q,b_1)$, by a homotopy in $St^{N_3}(im(\alpha))$. Inductively, $\alpha$ is freely homotopic to a $\mathcal Q$-loop $\alpha_k$ at the end point of $\beta$, by a homotopy in $St^{kN_3}(im(\alpha))\subset \tilde X-C$. As $\tilde X(Q,w)\cap C=\emptyset$ and $im(\alpha_k)\subset \tilde X(Q,w)$, $\alpha_k$ (and hence $\alpha$) is homotopically trivial in $\tilde X-C$.  
\end{proof}

\noindent {\bf Case 1: $Q$ is 1-ended.} There are finitely many vertices $w_1,\ldots, w_n\in \tilde X$ such that $\tilde X(Q,w_i)\cap St^{(d_1N_1+1)N_3}(C)\ne \emptyset$. 
As $\tilde X(Q,w_i)$ is 1-ended, there is a compact subcomplex $D$ of $\tilde X$ such that $St^{(d_1N_1+1)N_3}(C)\subset D$ and for all $i\in \{1,\ldots ,n\}$ and vertices $x,y\in \tilde X(Q,w_i)-D$, 
$x$ and $y$ can be joined by a $\mathcal Q$-edge path in $\tilde X(Q,v_i)-St^{(d_1N_1+1)N_3}(C)$. Now, suppose $\alpha$ is an arbitrary loop in $\tilde X-D$ with initial vertex $v$. Choose $L$ a positive integer such that if $q_1$ and $q_2$ are vertices of $\tilde X(Q,\ast)$ that are of distance $\leq N_1|\alpha|$ apart in $\tilde X$ then they are of distance $\leq L$ in $\tilde X(Q, \ast)$. Choose $E$ such that any edge path loop $\tau$ at a vertex $x$ of $\tilde X$, of length $\leq N_1|\alpha|+L$, is homotopically trivial in $St^E(x)$. 
Let $\beta_1$ be a $\mathcal Q$-path in $\tilde X(Q,v)-St^{(d_1N_1+1)N_3}(C)$ from $v$ to a point 
$$w\in \tilde X-(St^E(C)\cup St^{d_1N_1N_3+L}(C)\cup St^{N_1|\alpha |}(D))$$ 

Write the edge path $\alpha$ as $(e_1,\ldots, e_m)$ with consecutive vertices $v=v_0,\ldots, v_m$.
As $w\in \tilde X(Q,v)$ there is an edge path $\tau_1$ of length $\leq N_1$ from $w=w_1$ to $w_2\in \tilde X(Q, v_2)$. Let $\tau _2$ be an edge path of length $\leq N_1$ from $w_2$ to $w_3\in \tilde X(Q, v_3)$. Inductively, $\tau_m $ is an edge path of length $\leq N_1$ from $w_{m}$ to a vertex $w_{m+1}\in  \tilde X(Q,v)$. (Note that $\tau_i$ may be taken as the trivial path if $e_i$ is a $\mathcal Q$-edge.) As the edge path $(\tau_1,\ldots, \tau_m)$ has length $\leq N_1|\alpha |$, there is a $\mathcal Q$-path $\lambda$, from $w_{m+1}$ to $ w$ of length $\leq L$. 
By the definition of $E$, the loop $\tau\equiv (\tau_1,\ldots ,\tau_m,\lambda)$ at $w$ is homotopically trivial in $\tilde X-C$. Hence, it suffices to show that $\alpha $ is freely homotopic to $\tau$ in $\tilde X-C$. (See figure 1.)


\hspace{1.2in}\includegraphics[scale=.79]{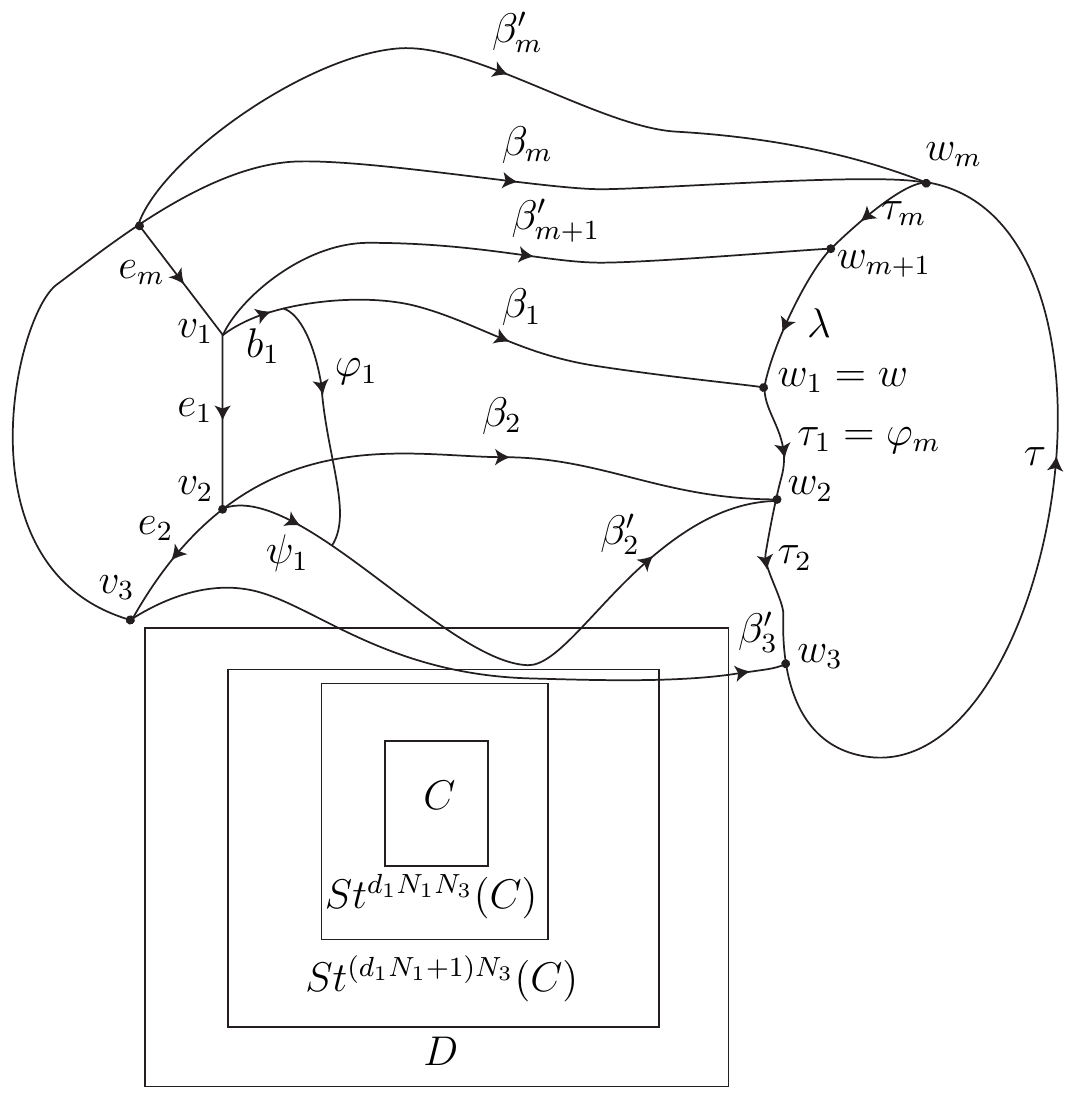}

\centerline{Figure 1}

\medskip

First note that each vertex of $(\tau_1,\ldots, \tau_m)$ is in $\tilde X-D$, since the vertex $w\in \tilde X-St^{N_1|\alpha |}(D)$.
Next, write $\beta _1$ as the edge path $(b_1,\ldots, b_s)$. Let $\phi_0=e_1$ and let $\phi_i$ be an edge path of length $\leq N_1$ from the end point of $b_i$ to a point of $\tilde X(Q,v_2)$. Let $\psi_i$ be a $\mathcal Q$-edge path of length $\leq N_2$ from the end point of $\phi_{i-1}$ to the end point of $\phi_i$. (Choose $\phi_s=\tau_1$.) Then the loop $(\phi_{i-1},\psi_i,\phi_i^{-1},b_i^{-1})$ has length $\leq 2N_1+N_2+1$ and is homotopically trivial by a homotopy in the $N_3$-star of the initial point of $b_i$. String together these homotopies and we have that
the edge path $\langle e_1^{-1},\beta_1,\tau_1)$ is homotopic $rel\{0,1\} $ to the $\mathcal Q$-edge path $\beta_2'\equiv (\psi_1, \ldots, \psi_m)$ by a homotopy with image in $St^{N_3}(im(\beta_1))\subset \tilde X-St^{d_1N_1N_3}(C)$.
By the definition of $D$, there is an $\mathcal Q$-edge path $\beta _2$ with the same end points as $\beta_2'$ and with image in $\tilde X-St^{(d_1N_1+1)N_3}(C)$. By lemma \ref{Qkill}, $\beta_2$ and $\beta_2'$ are homotopic $rel\{0,1\}$ by a homotopy in $\tilde X-C$. Continue inductively until $\beta_m$ and $\beta'_{m+1}$ are defined. Since $w\in \tilde X-St^{d_1N_1N_3+L}(C)$, the path $\lambda$ (of length $\leq L$) has image in $\tilde X-St^{d_1N_1N_3}(C)$.
By lemma \ref{Qkill}, the $\mathcal Q$-loop $(\beta'_{m+1},\lambda, \beta_1^{-1})$ is homotopically trivial in $\tilde X-C$.

\noindent {\bf Case 2: $\Lambda(S,Q,G)$ is 1-ended.}
The letters $N_1$, $N_2$ and $N_3$ remain as in case 1 and we recycle letters used for any  other constant. 

Given $C$ a finite subcomplex of  $\tilde X$. Consider $\rho(C)\subset \hat  \Lambda (S,Q,G)$. Choose $D$ a finite subgraph of $ \Lambda (S,Q,G)$ such that any two vertices of $\hat \Lambda (S,Q,G)-D$ can be connected by a path in $\hat  \Lambda (S,Q,G)-\rho(St^{N_3}(C))$. For each vertex $v$ of $D$ choose a path $\bar \alpha_v$ from $v$ to a vertex of $\hat  \Lambda (S,Q,G)-D$. If $v$ is a vertex of $\hat \Lambda (S,Q,G)-D$, let $\bar \alpha_v$ be the trivial path. Let $N$ be the length of the longest path $\bar \alpha_v$ for $v\in D$. If $v$ is a vertex of $\tilde X$ such that $\rho(v)\in D$ let $\alpha_v$ be an edge path of the form $(\beta_1,\ldots , \beta_m)$ where each $\beta_i$ has length $\leq N_1$ and $\rho(\beta_i)$ has the same end points as the $i^{th}$ edge of  $\bar \alpha_{\rho(v)}$ (so $|\alpha_v|\leq N_1N$).  In analogy with previous terminology, we call $\alpha_v$ an $N_1$-approximate lift of $\bar \alpha_{\rho(v)}$. If $\rho(v)\not \in D$, let $\alpha_v$ be the trivial path. 

Choose an integer $M$ such that if $v$ and $w$ are adjacent vertices of $St(D)$, then there is an edge path $\bar \alpha_{v,w}$ in $\hat \Lambda (S,Q,G)-\rho(St^{N_3}(C))$ of length $\leq M$ from the end point of $\bar\alpha_v$ to the end point of $\bar \alpha_w$. Choose an integer $B$ such that if $\beta$ is a $\tilde X$-edge path of length $\leq (2N+M)N_1+1$ connecting $\ast$ (the vertex of $\tilde X$ corresponding to the identity element of $G$)  to a vertex $q\in Q$ then there is a $\mathcal Q$-edge path of length $\leq B$ connecting $\ast$ to $q$. Choose an integer $A$ such that if $\beta$ is an edge path loop at $\ast$ of length $\leq (2N+M)N_1+B+1$ then $\beta$ is homotopically trivial in $St^{A}(\ast)$.

We next show: If $\beta$ is an edge path loop in $\tilde X-St^A(C)$, then $\beta$ is freely homotopic to a loop $\hat \beta$ by a homotopy in $\tilde X-C$ where $\hat \beta$ can be chosen so that for each vertex $v$ of $\hat \beta$, $\rho(v)\not \in \rho(St^{N_3}(C))$. If $e$ is a directed edge of $\tilde X$ or $ \Lambda (S,Q,G)$, with initial point $a$ and terminal point $b$, then let $[a,b]$ represent this edge. 
Suppose $\beta$ is the edge path $(d_1,d_2,\ldots ,d_n)$ with consecutive vertices $b_1,\ldots , b_{n+1}$. If (cyclically) neither $\rho(b_i)$ nor $\rho(b_{i+1})$ is in $D$, then let $\hat \beta_i$ be the single edge $d_i$. Otherwise, $\rho(b_i)$ and $\rho(b_{i+1})$ belong to $St(D)$. In this case, consider the edge path $\delta_{i}\equiv (\alpha_{b_i}^{-1},d_i, \alpha_{b_{i+1}})$ of $\tilde X$.  

If $\rho(b_i)\ne \rho(b_{i+1})$, the edge path $\bar \alpha_{\rho(b_i),\rho(b_{i+1})}$ joins the end points of $\rho(\delta_{i})$ and has length $\leq M$. Let $\alpha_i$ be an $N_1$-approximate lift of $\bar \alpha_{\rho(b_i),\rho(b_{i+1})}$ to the initial point of $\delta _{i}$ (otherwise, let $\alpha_i$ be the trivial path at the initial point of $\delta_i$). 

Note that the end point of $ \alpha_i$ and the end point of $\delta_{i}$ belong to the same left $Q$-coset. As the length of $( \alpha_i^{-1},\delta_{i})$ is $\leq (2N+M)N_1+1$ there is a $\mathcal Q$-edge path $\gamma_{i}$ of length $\leq B$ from  the initial point to the end point of $( \alpha_i^{-1},\delta_{i})$. The loop $(\gamma_i^{-1}, \alpha_i^{-1}, \delta_i)$ has length $\leq 2N+M+B+1$ and so is homotopically trivial in $\tilde X-C$ (by the definition of $A$). Let $\hat \beta_i=(\alpha_i,\gamma_i)$, for $i\in \{1,\ldots , n\}$. Let $\hat \beta$ be the loop $(\hat \beta_1,\ldots , \hat\beta_n)$. Combining homotopies shows that $\beta$ is freely homotopic to $\hat \beta$ by a homotopy in $\tilde X-C$. As $\rho(\alpha_i)$ avoids $\rho(St^{N_3}(C))$, $\rho(\hat \beta)$ avoids $\rho(St^{N_3}(C))$. (See figure 2.)

\vspace*{-1.5in}

\includegraphics[scale=.68]{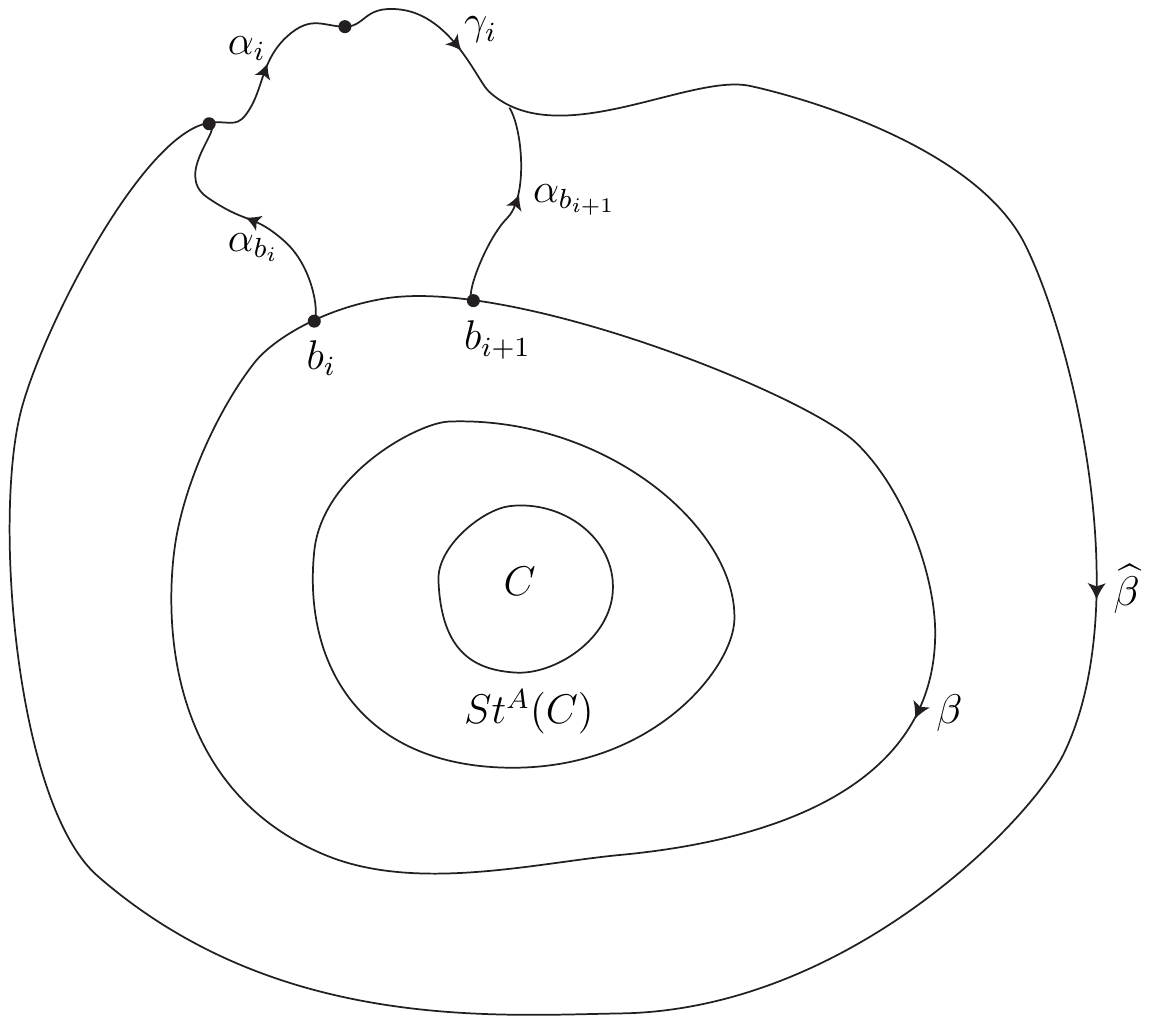}
\vspace*{-3in}

\centerline{Figure 2}

\medskip

We conclude the proof of case 2 by showing  $\hat \beta$ is homotopically trivial in $\tilde X-C$. The proof is analogous to the closing argument of case 1.
Let $v$ be the initial vertex of $\hat \beta$. Choose $L$ a positive integer such that if $q_1$ and $q_2$ are vertices of $\tilde X(Q,\ast)$ that are of distance $\leq N_1|\hat\beta|$ apart in $\tilde X$ then they are of distance $\leq L$ in $\tilde X(Q, \ast)$. Choose $E$ such that any edge path loop $\tau$ at a vertex $x$ of $\tilde X$ and  of length $\leq N_1|\hat\beta|+L$, is homotopically trivial in $St^E(x)$. 
Let $\beta_1$ be a $\mathcal Q$-path  from $v$ to a point 
$w\in \tilde X-St^E(C)$. 
Write the edge path $\hat\beta$ as $(e_1,\ldots, e_m)$ with consecutive vertices $v\equiv v_1,v_2,\ldots ,v_{m}$. 
As $w\in \tilde X(Q,v)$ there is an edge path $\tau_1$ of length $\leq N_1$ from $ w$ to $w_2\in \tilde X(Q, v_2)$. Let $\tau _2$ be an edge path of length $\leq N_1$ from $w_2$ to $w_3\in \tilde X(Q, v_3)$. Inductively, $\tau_m $ is an edge path of length $\leq N_1$ from $w_{m}$ to a vertex $w_{m+1}\in\tilde X(Q,v)$. (Note that $\tau_i$ may be taken as the trivial path if $e_i$ is a $\mathcal Q$-edge.) As the edge path $(\tau_1,\ldots \tau_m)$ begins and ends in $\tilde X(Q,v)$ and has length $\leq N_1|\hat\beta |$, there is a $\mathcal Q$-path $\lambda$, from $w_{m+1}$ to $ w$ of length $\leq L$. 
By the definition of $E$, the loop $\tau\equiv (\tau_1,\ldots ,\tau_m,\lambda)$ at $w$ is homotopically trivial in $\tilde X-C$. Hence, it suffices to show that $\alpha $ is freely homotopic to $\tau$ in $\tilde X-C$.

Each vertex $b$ of $\beta_1$ is such that $\rho(v)=\rho(b)\in \hat\Lambda (S,Q,G)-\rho(St^{N_3}(C))$ and so the image of $\beta_1$ avoids $St^{N_3}(C)$. As in case 1, this implies that the path $(e_1^{-1}, \beta _1,\tau_1)$ is homotopic $rel\{0,1\}$ to a $\mathcal Q$-edge path $\beta_2$ by a homotopy with image in $St^{N_3}(im(\beta_1))\subset \tilde X-C$. Each vertex $b$ of $\beta_2$ is such that $\rho(b)=\rho(v_2)\in \hat\Lambda(S,Q,G)-\rho(St^{N_3}(C))$ and so the image of $\beta_2$ avoids $St^{N_3}(C)$. 
The path $(e_2^{-1}, \beta _2,\tau_2)$ is homotopic $rel\{0,1\}$ to a $\mathcal Q$-edge path $\beta_3$ by a homotopy with image in $St^{N_3}(im(\beta_2))\subset \tilde X-C$. 
Continue inductively until $\beta_{m+1}$ is defined (as a $\mathcal Q$-path from $v$ to $w_{m+1}$). As $\rho(v)\in \hat\Lambda (S,Q,G)-\rho(St^{N_3}(C))$, the $\mathcal Q$-loop $(\beta_1,\lambda^{-1},\beta_{m+1}^{-1})$   has image in $\tilde X(Q,v)\subset \tilde X-C$,
and so is homotopically trivial in $\tilde X-C$. (See figure 3.)

\vspace*{-1.5in}
\includegraphics[scale=.68]{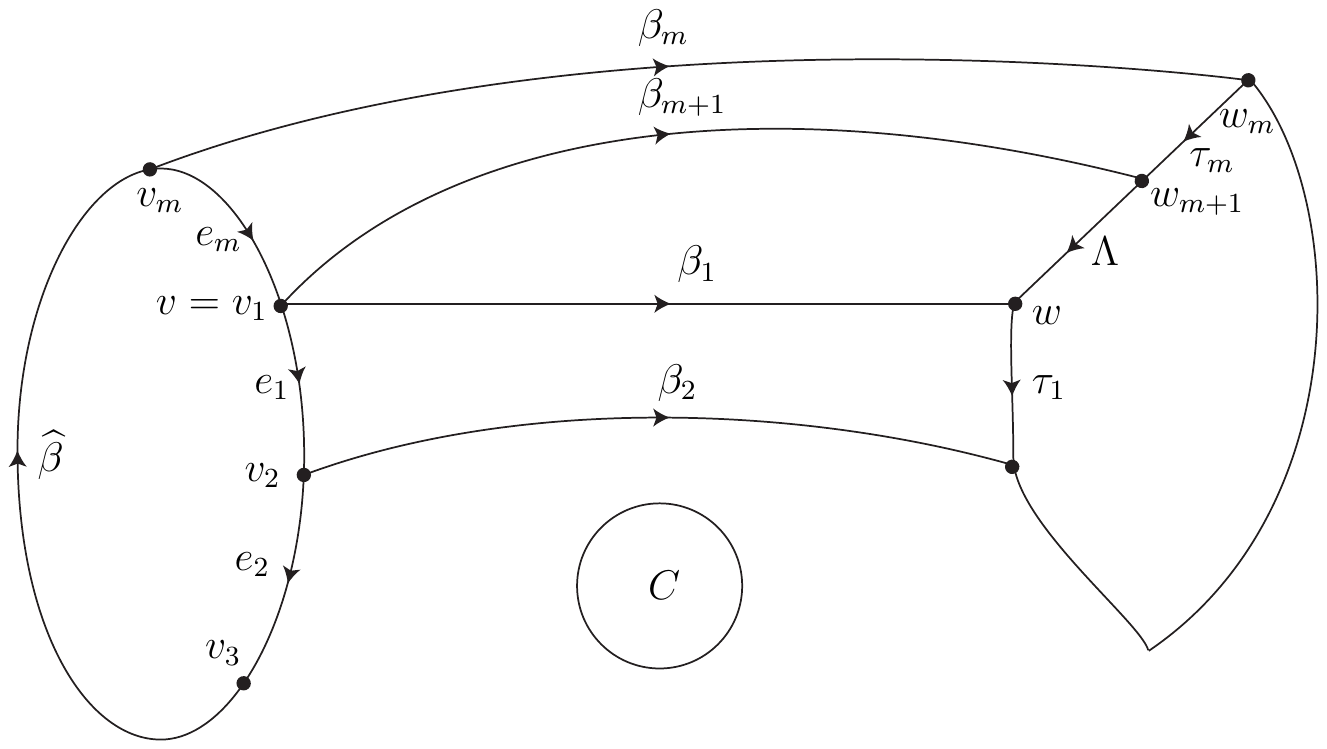}
\vspace*{-3.5in}

\centerline{Figure 3}

\medskip

Combining homotopies produces a null homotopy of $\hat \beta $ with image in $\tilde X-C$.
\end{proof}

\enddocument

DEPARTMENT OF MATHEMATICS WHITE HALL CORNELL UNIVERSITY ITHACA, NY 14853
CURRENT ADDRESS:
DEPARTMENT OF MATHEMATICS AND COMPUTER SCIENCE
STARK LEARNING CENTER
WILKES UNIVERSITY
WILKES-BARRE, PA 18766

\end{document}